\documentclass[11pt]{article}

\usepackage{array}
\usepackage{amsfonts}
\usepackage{amssymb,amsmath,amsthm}
\usepackage{graphics}
\usepackage{epsfig}
\usepackage{psfrag}
\usepackage{dsfont}

\usepackage[utf8]{inputenc}
 \usepackage[T1]{fontenc}

\numberwithin{equation}{section}

\def\QSD{quasi-stationary distribution}

\def\11{\mathds{1}}

\def\me{\medskip \noindent}
\def\bi{\bigskip \noindent}

\def\E{\mathbb{E}}
\def\P{\mathbb{P}}
\def\R{\mathbb{R}}

\def\N{\mathbb{N}}

\def\d{\partial}

\def\cX{{\cal X}}
\def\cM{{\cal M}}
\def\cD{{\cal D}}
\def\cL{{\cal L}}

\def\b1{\mathbf{1}}

\newtheorem{thm}{Theorem}[section]
\newtheorem{lem}[thm]{Lemma}

\theoremstyle{remark}
\newtheorem{rem}{Remark}
\newtheorem{exa}{Example}

\begin{document}

\title{Minimal quasi-stationary distribution approximation for a birth and death process}

\author{Denis Villemonais$^{1,2}$}

\footnotetext[1]{Universit\'e de Lorraine, IECN, Campus Scientifique, B.P. 70239,
  Vand{\oe}uvre-l\`es-Nancy Cedex, F-54506, France} \footnotetext[2]{Inria, TOSCA team,
  Villers-l\`es-Nancy, F-54600, France.\\
  E-mail: Denis.Villemonais@univ-lorraine.fr}

\maketitle

\begin{abstract}
In a first part, we prove a Lyapunov-type criterion for the $\xi_1$-positive recurrence of absorbed birth and death processes and provide new results on the domain of attraction of the minimal quasi-stationary distribution.
In a second part, we study the ergodicity and the convergence of a Fleming-Viot type particle system whose particles evolve independently as a birth and death process and jump on each others when they hit $0$. Our main result is that the sequence of empirical stationary distributions of the particle system converges to the minimal quasi-stationary distribution of the birth and death process.
\end{abstract}



\section{Introduction}
\label{sec:intro}

Let $X$ be a stable birth and death process on $\N=\{0,1,2,\ldots\}$ absorbed when it hits $0$.
The \textit{minimal quasi-stationary distribution} (or \textit{Yaglom limit}) of $X$, when it exists, is the unique probability measure $\rho$ on $\N^*=\{1,2,\ldots\}$ such that
\begin{align*}
\rho(\cdot)=\lim_{t\rightarrow\infty}\P_{x}\left(X_t\in\cdot\mid t<T_0\right),\ \text{for all }x\in\N^*,
\end{align*}
 where $T_0=\inf\{t\geq 0,\,X_t=0\}$ is the absorption time of $X$. 
The probability measure $\rho$ is called a \QSD{} because it is stationary for the conditioned process, in the sense that
\begin{align*}
\rho=\P_\rho(X_t\in\cdot \mid t < T_0), \text{ for all } t\geq 0.
\end{align*}
 These notions and important references on the subject are recalled with more details in Section~2, with important definitions and well known results on quasi-stationary distributions. We also provide a new Lyapunov-type criterion ensuring that a probability measure $\mu$ belongs to the \textit{domain of attraction} of the minimal \QSD, which means that
\begin{align}
\lim_{t\rightarrow\infty}\P_\mu\left(X_t\in\cdot\mid t<T_0\right)=\rho.
\end{align}
These results are illustrated with several examples.

\bi
We use these new results in Section~\ref{sec:FV} to extend existing studies on the long time and high number of particles limit of a \textit{Fleming-Viot type particle system}. 
 The particles of this system evolve as independent copies of the birth and death process $X$, but they undergo rebirths when they hit $0$ instead of being trapped at the origin. In particular, the number of particles that are in $\N^*$ remains constant as time goes on. Our main result is a sufficient criterion ensuring that the empirical stationary distribution of the particle system exists and converges to the minimal \QSD{} of the underlying birth and death process.

\bi We conclude the paper in Section~4, providing a numerical study of the speed of convergence of the Fleming-Viot empirical stationary distribution expectation to the minimal \QSD{} for a linear birth and death process and a logistic birth and death process.  This numerical results suggest that the bias of the approximation is surprisingly small for linear birth and death processes and even smaller for logistic birth an death processes.

\section{Quasi-stationary distributions for birth and death processes}

Let $(X_t)_{t\geq 0}$ be a birth and death process on $\N=\{0,1,2,\ldots\}$ with birth rates $(b_i)_{i\geq 0}$ and death rates $(d_i)_{i\geq 0}$. We assume that $b_i>0$ and $d_i>0$ for any $i\geq 1$ and $b_0=d_0=0$. The stochastic process $X$ is a $\N$-valued pure jump process whose only absorption point is $0$ and whose transition rates from any point $i\geq 1$ are given by
\begin{align*}
i&\rightarrow i+1\text{ with rate }b_i,\\
i&\rightarrow i-1\text{ with rate }d_i,\\
i&\rightarrow j\text{ with rate }0,\text{ if }j\notin\{i-1,i+1\}.
\end{align*}
Such processes are extensively studied because of their conceptual simplicity and pertinence as demographic models. It is well known (see for instance \cite[Theorem~10 and Proposition~12]{MV12}) that $X$ is stable, conservative and hits $0$ in finite time almost surely (for any initial distribution) if and only if
\begin{align}
\label{eq:as_absorption}
\sum_{n=1}^{\infty} \frac{d_1d_2\cdots d_n}{b_1b_2\cdots b_n}=+\infty.
\end{align}
The divergence of this series will be assumed along the whole paper. In particular, for any probability measure $\mu$ on $\N$, the law of the process with initial distribution $\mu$ is well defined. We denote it by $\P_{\mu}$ (or by $\P_x$ if $\mu=\delta_x$ with $x\in\N$) and the associated expectation by $\E_{\mu}$ (or by $\E_x$ if $\mu=\delta_x$ with $x\in\N$). Setting $T_0=\inf\{t\geq 0,\ X_t=0\}$, we thus have 
\begin{align*}
\P_{\mu}(T_0<\infty)=1,\ \forall \mu\in{\cal M}_1(\N),
\end{align*}
where, for any subset $F\subset \N$, ${\cal M}_1(F)$ denotes the set of probability measures on $F$. 

\me 
A \textit{quasi-stationary distribution} for $X$ is a probability measure $\rho$ on $\N^*=\{1,2,\ldots\}$ such that
\begin{align*}
\P_{\rho}\left(X_t\in\cdot\mid t< T_0\right)=\rho(\cdot),\ \forall t\geq 0.
\end{align*}
The probability measure $\rho$ is thus stationary for the conditioned process (and, as a matter of fact, was called a \textit{stationary distribution} in the seminal work~\cite{Cavender1978}).  The property "$\rho$ is a quasi-stationary distribution for $X$" is directly related to the long time behaviour of $X$ conditioned to not being absorbed. Indeed (see for instance~\cite{vanDoorn1991} or~\cite{MV12}), 
 a probability measure $\rho$ is a quasi-stationary distribution if and only if there exists $\mu\in{\cal M}_1(\N^*)$ such that
\begin{align}
\label{eq:QLD}
\rho(\cdot)=\lim_{t\rightarrow\infty}\P_{\mu}(X_t\in\cdot\mid t<T_0).
\end{align}

\me
We refer the reader to~\cite{vanDoorn_Pollett_2013,MV12,ColletMartinezsanMartin2013} and references therein for an account on classical results concerning \QSD{}s for different models.

\me
 For a given quasi-stationary distribution $\rho$, the set of probability measures $\mu$ such that~\eqref{eq:QLD} holds is called the \textit{domain of attraction of $\rho$}. It is non-empty since it contains at least $\rho$ and may contains an infinite number of elements. In particular, when the limit in~\eqref{eq:QLD} exists for any $\mu=\delta_x$, $x\in\N^*$,  and doesn't depend on the initial position $x$, then $\rho$ is called the \textit{Yaglom limit} or the \textit{minimal quasi-stationary distribution}. Thus the minimal \QSD, when it exists, is the unique quasi-stationary distribution whose domain of attraction contains $\{\delta_x,\ x\in\N^*\}$. From a demographical point of view, the study of the minimal \QSD{} of a birth and death process aims at answering the following question: \textit{knowing that a population isn't extinct after a long time $t$, what is the probability that its size is equal to $n$ at time~$t$?}

\bi One of the oldest and most understood question for quasi-stationary distributions of birth and death processes concerns their existence and uniqueness. Indeed, van Doorn~\cite{vanDoorn1991} gave the following picture of the situation: a birth and death process can have no quasi-stationary distribution, one unique quasi-stationary distribution or an infinity (in fact a continuum) of quasi-stationary distributions. In order to determine whether a birth and death process has $0$, one or an infinity of quasi-stationary distributions, one define inductively the sequence of polynomials $(Q_n(x))_{n\geq 0}$  for all $x\in\R$ by
\begin{align}
\label{eq:polynomials}
\begin{array}{l}
\displaystyle
Q_1(x)=1,\\
\displaystyle b_1\,Q_2(x) = b_1+d_1 - x\text{ and }\\
\displaystyle b_n\, Q_{n+1}(x) = (b_n + d_n - x)\,Q_n(x) -
d_{n-1}\,Q_{n-1}(x),\ \forall n\geq 2.
\end{array}
\end{align} 
As recalled in~\cite[eq. (2.13)]{vanDoorn1991}, one can uniquely define the non-negative number $\xi_{1}$ satisfying
\begin{align}
\label{eq:def_xi_1}
  x\leq \xi_1 \
  \Longleftrightarrow\ Q_n(x) >0, \ \forall n\geq 1.
\end{align}
Also, the useful quantity
  \begin{equation*}
    S:=\sup_{x\geq 1} \E_x(T_0),
  \end{equation*}
can be easily computed (see~\cite[Section~8.1]{Anderson1991}), since, for any $z\geq 1$,
\begin{align*}
\sup_{x\geq z} \E_x(T_z)= \sum_{k\geq z+1}\frac{1}{d_k\pi_k}\sum_{l\geq k} \pi_l,
\end{align*}
with
$
    \pi_k=\left(\prod_{i=1}^{k-1} b_i\right)/\left(\prod_{i=2}^{k} d_i\right).
$
The following theorem answers the question of existence and uniqueness of a QSD for birth and death processes.
\begin{thm}[van Doorn, 1991~\cite{vanDoorn1991}]
\label{thm:van-doorn}
Let $X$ be a birth and death process satisfying~\eqref{eq:as_absorption}.
  \begin{enumerate}
       \item If $\xi_1=0$, there is no
          QSD.    
      \item  If $S<+\infty$, then $\xi_1>0$ and the Yaglom limit is the unique QSD.
      \item If $S=+\infty$ and $\xi_1> 0$, then there is a
        continuum of QSDs, given by the one parameter family $(\rho_x)_{0<x\leq  \xi_1}$:
	\begin{equation*}
	  \rho_x(j)=  \frac{\pi_j}{d_1}\, x\, Q_j(x),\ \forall j\geq 1,
	\end{equation*}
	and the minimal \QSD{} is given by $\rho_{\xi_1}$.
   \end{enumerate}
\end{thm}

\begin{rem}
 Theorem~\ref{thm:van-doorn} gives a complete description of the set of \QSD s for a birth and death process but is not well suited for the numerical computation of the Yaglom limit of a given birth and death process. Indeed, the polynomials $Q_n$ have in most cases quickly growing coefficients, so that the value of $\xi_1$ cannot be easily obtained by numerical computation.
\end{rem}

\bi Theorem~\ref{thm:van-doorn} is quite remarkable since it describes completely the possible outcomes of the existence and uniqueness problem for quasi-stationary distributions. However, it only partially answers the crucial problem of finding the domain of attraction of the existing quasi-stationary distributions and in particular of the minimal \QSD. The following theorem answers the problem when there exists a unique quasi-stationary distribution.

\begin{thm}[Mart\'inez, San Mart\'in, Villemonais 2013~\cite{Martinez-Martin-Villemonais2013}]
\label{thm:martinez-san-martin-villemonais}
Let $X$ be a birth and death process such that
\begin{align*}
S=\sup_{x\geq 1} \E_x(T_0)<+\infty.
\end{align*}
Then there exists $\gamma\in[0,1)$ such that, for any probability measure $\mu$ on $\N^*$,
\begin{align*}
\left\|\rho-\P_{\mu}\left(X_t\in\cdot\mid t<\tau_\d\right)\right\|_{TV}\leq \gamma^{\lfloor t\rfloor},\ \forall t\geq 0,
\end{align*}
where $\|\cdot\|_{TV}$ denotes the total variation norm and $\rho$ is the unique \QSD{} of the process. In particular, the domain of attraction of the unique \QSD{} is the whole set $\cM_1(\N^*)$ of probability measures on $\N^*$.
\end{thm}

\noindent A weaker form of Theorem~\ref{thm:martinez-san-martin-villemonais} has also been proved in~\cite{Zhang2013} but the strong form (with uniform convergence in total variation norm) is necessary to derive the results of the next section. A generalized version of Theorem~\ref{thm:martinez-san-martin-villemonais} has been rencently derived in~\cite{ChampagnatVillemonais2014}, with complementary results on the so-called $Q$-process (the process conditioned to never being absorbed).

\bi 
The case where there exists an infinity of quasi-stationary distributions is trickier and can be partially solved, as we will show, when the birth and death process is $\xi_1$-positive recurrent.

\me
\textbf{Definition}
The birth and death process $X$ is said to be \textit{$\xi_1$-positive recurrent} if $\xi_1>0$ in Theorem~\ref{thm:van-doorn} and if, for some $i\in\{1,2,\ldots\}$ and hence for all $i\in\{1,2,\ldots\}$, we have
\begin{align*}
\lim_{t\rightarrow\infty} e^{\xi_1 t}\,\P_i(X_t=i)>0.
\end{align*}

\me In the following theorem, we provide a new Lyapounov--type criterion ensuring the $\xi_1$-positive recurrence of a birth and death process. As will be shown in the examples below, this criterion can be checked on a wide variety of examples and has its own interest in the domain of $\xi_1$-classification for birth and death processes (see~\cite{HartMartinezsanMartin2003} and~\cite{vanDoorn2006} for an account on this area).

\begin{thm}
\label{thm:positivity_criterion}
Let $X$ be a birth and death process with infinitesimal generator $\cL$. We assume that there exists $C>0$, $\lambda_1>d_1$ and $\phi:\N\rightarrow\R_+$  such that $\phi(i)$ goes to infinity when $i\rightarrow\infty$ and
\begin{align*}
\cL \phi(i)\leq -\lambda_1 \phi(i) + C,\ \forall i\geq 1.
\end{align*}
Then $X$ admits a quasi-stationary distribution and the birth and death process $X$ is $\xi_1$-positive recurrent.
\end{thm}

\bi In the next theorem, we assume that the process is $\xi_1$-positive recurrent and we exhibit a subset of the domain of attraction for the minimal \QSD. 

\begin{thm}
\label{thm:minimal-QSD-attraction}
Let $X$ be a $\xi_1$-positive recurrent birth and death process with infinitesimal generator $\cL$.
 Then the domain of attraction of the minimal \QSD{} of $X$ contains the set $\cD$ defined by
\begin{align*}
\cD=\left\{\mu\in\cM_1(\N^*),\, \sum_{i=1}^{\infty} \mu_i Q_i(\xi_1)<+\infty\right\}.
\end{align*}
Assume moreover that there exists $C>0$, $\lambda_1>\xi_1$ and $\phi:\N\rightarrow\R_+$  such that $\phi(i)$ goes to infinity when $i\rightarrow\infty$ and
\begin{align*}
\cL \phi(i)\leq -\lambda_1 \phi(i) + C,\ \forall i\geq 1.
\end{align*}
Then the domain of attraction of the minimal \QSD{} of $X$ contains the set $\cD_\phi$ defined by
\begin{align*}
\cD_{\phi}=\left\{\mu\in\cM_1(\N^*),\, \sum_{i=1}^{\infty} \mu_i \phi(i)<+\infty\right\}.
\end{align*}
\end{thm}

\bi As it will be shown in the proof, we have $\cD_\phi\subset \cD$ for all function $\phi$ satisfying the assumptions of Theorem~\ref{thm:minimal-QSD-attraction}. However, $Q_\cdot(\xi_1)$ cannot be computed explicitly but in few situations. As a consequence, we won't be able to use the first criterion to determine whether a probability distribution $\mu$ belongs or not to the domain of attraction of the minimal \QSD. On the contrary, we will be able to give explicit functions $\phi$ satisfying the Lyapunov criterion of our theorem for a wide range of situations.

\bi Note that, since $d_1\geq \xi_1$, our results immediately imply that, if the process $X$ fulfils the assumptions of Theorem~\ref{thm:positivity_criterion} with a Lyapunov function $\phi$, then the process is $\xi_1$-positive recurrent and the domain of attraction of its minimal \QSD{} contains $\cD_\phi$. This consequence is used in the following examples.
\bi

\begin{exa}
We consider the case where $b_i=b\,i^a$ and $d_i=d\,i^a$ for all $i\geq 1$, where $b<d$ are two positive constants and $a>0$ is fixed.
Now, defining $\phi(0)=0$ and
\begin{align*}
\phi(i)=\sqrt{d/b}^{\,i},\ \forall i\in\N^*,
\end{align*}
one gets 
\begin{align*}
\cL\phi(i)&:=b_i(\phi(i+1)-\phi(i))+d_i(\phi(i-1)-\phi(i))\\
&= b\,i^a\left(\sqrt{d/b}^{\,i+1}-\sqrt{d/b}^{\,i}\right)+d\,i^a\left(\sqrt{d/b}^{\,i-1}-\sqrt{d/b}^{\,i}\right)\\
&= i^a \left[\left(\sqrt{db}\sqrt{d/b}^{\,i}-b\sqrt{d/b}^{\,i}\right)+\left(\sqrt{db}\sqrt{d/b}^{\,i}-d\sqrt{d/b}^{\,i}\right)\right]\\
&= -i^a\left(\sqrt{d}-\sqrt{b}\right)^2\sqrt{d/b}^i.
\end{align*}
 Since $i^a\rightarrow\infty$ when $i\rightarrow\infty$, we immediately deduce that there exists $C>0$ and $\lambda_1>d_1$ such that $\phi$ satisfies $\cL\phi\leq -\lambda_1 \phi+C$. Now
 Theorem~\ref{thm:positivity_criterion} implies that the process is $\xi_1$-positive recurrent and Theorem~\ref{thm:minimal-QSD-attraction} implies that the domain of attraction of the minimal \QSD{} contains
\begin{align*}
\cD_{\phi}=\left\{\mu\in\cM_1(\N^*),\, \sum_{i=1}^{\infty} \mu_i \sqrt{d/b}^i<+\infty\right\}.
\end{align*}
\end{exa}

\bi
\begin{exa}
We consider now the case where the birth and death rates are constant for all $i\geq 2$, that is $b_i=b>0$  and $d_i=d>0$ for all $i\geq 2$, where $b<d$ are positive constants. We assume that $d_1>0$ is such that $(\sqrt{d}-\sqrt{b})^2> d_1$ and the value of $b_1>0$ can be chosen arbitrarily.
 Using the same function as in the previous example, that is
\begin{align*}
\phi(i)=\sqrt{d/b}^{\,i},\ \forall i\geq 2,
\end{align*}
one gets 
\begin{align*}
\cL\phi(i)= -(\sqrt{d}-\sqrt{b})^2\sqrt{d/b}^i,\ \forall i\geq 2.
\end{align*}
In particular, there exists $C>0$ and $\lambda_1>d_1$ such that $\phi$ satisfies $\cL\phi\leq -\lambda_1 \phi+C$. Once again, we deduce from Theorem~\ref{thm:positivity_criterion} that the process is $\xi_1$-positive recurrent, which was already known in this case (see~\cite[eq.~(6.6)]{vanDoorn2003}). We also deduce the following new result from Theorem~\ref{thm:minimal-QSD-attraction}: the domain of attraction of the minimal \QSD{} contains the set
\begin{align*}
\cD_{\phi}=\left\{\mu\in\cM_1(\N^*),\, \sum_{i=1}^{\infty} \mu_i \sqrt{d/b}^i<+\infty\right\}.
\end{align*}
\end{exa}

\begin{exa}
In the two previous examples, the birth and death rates are non-decreasing and proportional to each other. This is coincidental and is only useful to get straightforward calculations. The aim of the present example is to illustrate this on a particular case without monotony nor proportionality between the birth and death rates: we choose $b_i=|\sin(i\pi/2)|i+1$ and $d_i=4i$ for all $i\geq 1$. Now, defining 
\begin{align*}
\phi(i)=2^i,\ \forall i\geq 1,
\end{align*}
we get, for all $i\geq 2$,
\begin{align*}
\cL \phi(i) = &:=b_i(\phi(i+1)-\phi(i))+d_i(\phi(i-1)-\phi(i))\\
&=(|\sin(i\pi/2)|i+1)\left(2^{i+1}-2^i\right)+4i(2^{i-1}-2^i)\\
&\leq 2^i\left(i+1-2i\right)=-(i-1)\phi(i).
\end{align*}
As above, we deduce that the process is $\xi_1$-positive recurrent and that the domain of attraction of the minimal \QSD{} contains the set of probability measures defined by $\left\{\mu\in\cM_1(\N^*),\, \sum_{i=1}^{\infty} \mu_i 2^i<+\infty\right\}$.
\end{exa}

\bi
The end of this section is dedicated to the proof of Theorems~\ref{thm:positivity_criterion} and~\ref{thm:minimal-QSD-attraction}.

\begin{lem}
\label{lem:useful1}
We assume that there exists $\lambda_1>\xi_1$, $C>0$ and $\phi:\N\rightarrow\R_+$  such that $\phi(i)$ goes to infinity when $i\rightarrow\infty$ and
\begin{align*}
\cL \phi(i)\leq -\lambda_1 \phi(i) + C,\ \forall i\geq 1.
\end{align*}
Then there exists a constant $i_0\geq 1$ and $\alpha>0$ such that $\phi(i)\geq \alpha Q_{i}(\xi_1)$, for all $i\geq i_0$.
\end{lem}

\begin{proof}
 Since $\phi(i)\rightarrow\infty$ when $i\rightarrow\infty$, there exists $i_0\geq 1$ such that $\phi(i)\geq C/\lambda_1+1$ for any $i\geq i_0-1$. Let us set $$\phi'(i)=(\phi(i)-C/\lambda_1)_+\text{ for any }i\in\N^*.$$
 For all $i\geq i_0$, we have $\phi'(i)\geq 1$  and
\begin{align*}
\cL\phi'(i)&=\cL\phi(i)\leq -\lambda_1\phi(i)+C \nonumber\\
&\leq -\lambda_1(\phi'(i)+C/\lambda_1)+C\\
&=-\lambda_1 \phi'(i)
\end{align*}
Hence, since $Q_{i_0}(\xi_1)>0$ (see \cite[eq. (2.13)]{vanDoorn1991})  and replacing $\phi$ by $Q_{i_0}(\xi_1)\phi'$, we can assume without loss of generality that
\begin{align*}
\phi(i_0)\geq Q_{i_0}(\xi_1)\text{ and }\cL(\phi)\leq -\lambda_1\phi,\ \forall i\geq i_0.
\end{align*}
Our aim is now to prove that, for all $i\geq i_0$, $\phi(i)\geq Q_{i}(\xi_1)$. Because of the changes made with the function $\phi$, this implies the inequality claimed in the lemma with the constant $\alpha=1/Q_{i_0}(\xi_1)$.

Assume the contrary, which means that there exists $i_1\geq i_0$ such that $\phi(i_1)\geq Q_{i_1}(\xi_1)$ and $\phi(i_1+1)< Q_{i_1+1}(\xi_1)$.
Now, fix $x\in (\xi_1,\lambda_1]$ such that
\begin{align*}
Q_{i_1+1}(x)\frac{Q_{i_1}(\xi_1)}{Q_{i_1}(x)}>\phi(i_1+1).
\end{align*}
This is feasible, since $x\mapsto Q_{i}(x)$ is a polynomial function of $x$ and is then continuous in $x$ for any fixed $i$.
Then the function $\varphi_x:\N\rightarrow \R_+$ defined by
\begin{align*}
\varphi_x(i)=Q_{i}(x)\frac{Q_{i_1}(\xi_1)}{Q_{i_1}(x)}
\end{align*}
satisfies, $\cL \varphi_x(i)=-x\varphi_x(i)$ for all $i\in \N$, $\varphi_x(i_1)=Q_{i_1}(\xi_1)\leq \phi(i_1)$ and $\varphi_x(i_1+1)>\phi(i_1+1)$.

\me
Let us now prove that this inequality extends to any $j\geq i_1+1$. Indeed, using the equality $\cL \varphi_x=-x\varphi_x$ and the inequality $\cL \phi \leq -\lambda_1 \phi$, we have, for all $j> i_1+1$,
\begin{align*}
\varphi_x(i_1+1)&=\E_{i_1+1} \left(e^{ \,x T_j\wedge T_{i_1}}\varphi_x(X_{T_j\wedge T_{i_1}})\right)\\
&=\E_{i_1+1} \left(e^{ \,x T_j}\11_{T_j< T_{i_1}}\right)\varphi_x(j)+\E_{i_1+1} \left(e^{ \,x T_{i_1}}\11_{ T_{i_1}<T_j}\right)\varphi_x(i_1)
\end{align*}
and
\begin{align*}
\phi(i_1+1)&\geq \E_{i_1+1} \left(e^{ \,\lambda_1 T_j\wedge T_{i_1}}\phi(X_{T_j\wedge T_{i_1}})\right)\\
&= \E_{i_1+1} \left(e^{ \,\lambda_1 T_j}\11_{T_j< T_{i_1}}\right)\phi(j)+\E_{i_1+1} \left(e^{ \,\lambda_1 T_{i_1}}\11_{ T_{i_1}<T_j}\right)\phi(i_1)\\
&\geq \E_{i_1+1} \left(e^{ \,x T_j}\11_{T_j< T_{i_1}}\right)\phi(j)+\E_{i_1+1} \left(e^{ \,x T_{i_1}}\11_{ T_{i_1}<T_j}\right)\varphi_x(i_1).
\end{align*}
We deduce that
\begin{align*}
\E_{i_1+1} \left(e^{ \,x T_j}\11_{T_j< T_{i_1}}\right)\left(\phi(j)-\varphi_x(j)\right)\leq \phi(i_1+1)-\varphi_x(i_1+1)< 0
\end{align*}
and thus $\phi(j)<\varphi_x(j)$ for all $j\geq i_1+1$.

As a direct consequence, we deduce that $Q_n(x)>0$ for all $n\geq 0$. But $x>\xi_1$, which is in contradiction with~\cite[eq. (2.13)]{vanDoorn1991}.
\end{proof}

\begin{lem}
\label{lem:usefull2}
Let $\phi:\N\rightarrow\R_+$ be a function such that $\phi(0)=0$, $\phi(1)=1$ and $\cL\phi\leq 0$. Then 
\begin{align*}
\min_{i\geq 1} \phi(i)=1.
\end{align*}
\end{lem}

\begin{proof}
Under our assumption, $(\phi(X_t))_{t\geq 0}$ is a super martingale. As a consequence, for all $i\geq 1$,
\begin{align*}
\E_i\left(\phi(X_{T_1})\11_{T_1<\infty}\right)\leq \phi(i).
\end{align*}
But $T_1\leq T_0<\infty$ almost surely and $X_{T_1}=1$ almost surely, so that
\begin{align*}
1\leq \phi(i),\ \forall i\geq 1.
\end{align*}
\end{proof}

\bi\begin{proof}[Proof of Theorem~\ref{thm:positivity_criterion}]
The main difficulty is to prove that the minimal quasi-stationary distribution $\rho_{\xi_1}$ for $X$ exists and that 
\begin{align}
\label{eq:bound_on_rho_xi_1}
\rho_{\xi_1}(\phi)\leq \phi(1)\vee \frac{C}{\lambda_1-d_1}.
\end{align}
Once this is proved, Lemma~\ref{lem:useful1} implies that
\begin{align*}
\sum_{i\geq 1}\rho_{\xi_1}(i) Q_i(\xi_1)<\infty,
\end{align*}
which is a sufficient condition for $X$ to be $\xi_1$-positive recurrent (see~\cite[Theorem~5.2]{vanDoorn2003}).

Let us prove that~\eqref{eq:bound_on_rho_xi_1} holds.
For any $M\geq 1$, let us denote by $(P_t^M)_{t\geq 0}$ the semi-group of the process $X^M$ evolving in $\{0,1,\ldots,M\}$ and defined as
\begin{align*}
X^M_t=X_{t\wedge T_M},\ \forall t\geq 0,
\end{align*}
where $T_M=\inf\{t\geq 0,\ X_t=M\}$. We also define $\phi^M$ by $\phi^M(M)=0$ and $\phi^M(i)=\phi(i)$, $i\in\{0,1,\ldots,M-1\}$.
Now, denoting by $\cL^M$ the generator of the stopped process $X^M$ and setting $\varphi(i)=\11_{i\geq 1}$, we thus have
\begin{align*}
\cL^M \phi^M(i) \leq -\lambda_1 \phi^M(i) + C\varphi(i),\ \forall i\in\{0,1,\ldots,M\}
\end{align*}
and
\begin{align*}
\cL^M \varphi(i) \geq -d_1 \varphi(i),\ \forall i\in\{0,1,\ldots,M\}.
\end{align*}
Hence, using Kolmogorov equations for the finite state space continuous time Markov chain $X^M$, we deduce that
\begin{align*}
\frac{d}{dt}\left(\frac{\E_1\left(\phi^M(X^M_t)\right)}{\E_1\left(\varphi(X^M_t)\right)}\right)
&\leq -\lambda_1\frac{\E_1\left(\phi^M(X^M_t)\right)}{\E_1\left(\varphi(X^M_t)\right)}+C+d_1\frac{\E_1\left(\phi^M(X^M_t)\right)}{\E_1\left(\varphi(X^M_t)\right)}.
\end{align*}
This implies that, for any $t\geq 0$, we have
\begin{align*}
\frac{\E_1\left(\phi^M(X^M_t)\right)}{\E_1\left(\varphi(X^M_t)\right)}\leq \frac{\E_1\left(\phi^M(X^M_0)\right)}{\E_1\left(\varphi(X^M_0)\right)}\vee \frac{C}{\lambda_1-d_1}.
\end{align*}
But $X_0^M=1$ under $\P_1$, so that
\begin{align*}
\frac{\E_1\left(\phi^M(X^M_t)\right)}{\E_1\left(\varphi(X^M_t)\right)}\leq \phi^M(1)\vee \frac{C}{\lambda_1-d_1}=\phi(1)\vee \frac{C}{\lambda_1-d_1}.
\end{align*}
Now, by dominated convergence, we have
\begin{align*}
\E_1\left(\varphi(X^M_t)\right)\xrightarrow[M\rightarrow\infty]{} \E_1\left(\varphi(X_t)\right)=\P(X_t\neq 0).
\end{align*}
By monotone convergence, we also deduce that
\begin{align*}
\E_1\left(\phi^M(X^M_t)\right)\xrightarrow[M\rightarrow\infty]{} \E_1\left(\phi(X_t)\right).
\end{align*}
We finally deduce that, for all $t\geq 0$, 
\begin{align*}
\E_1\left(\phi(X_t)\mid X_t\neq 0\right)\leq \phi(1)\vee \frac{C}{\lambda_1-d_1}.
\end{align*}

The first consequence of this inequality is that the process $X_t$ conditioned on the event $X_t\neq 0$ does not diverge to infinity. As a consequence, $\xi_1>0$ and there exists a minimal quasi-stationary distribution $\rho_{\xi_1}$ for $X$ (see \cite[Theorem~4.1]{vanDoorn1991}). In particular, $X_t$ conditioned on $X_t\neq 0$ converges in law to $\rho_{\xi_1}$. Hence, we deduce from the above inequality that, for any $K\geq 0$,
\begin{align*}
\rho_{\xi_1}(\phi\wedge K)\leq \phi(1)\vee \frac{C}{\lambda_1-d_1}.
\end{align*}
By monotone convergence, we obtain by letting $K$ tend to $\infty$ that
\begin{align*}
\rho_{\xi_1}(\phi)\leq \phi(1)\vee \frac{C}{\lambda_1-d_1}.
\end{align*}
\end{proof}

\bi
\begin{proof}[Proof of Theorem~\ref{thm:minimal-QSD-attraction}]
Let $X$ be a $\xi_1$-positive recurrent birth and death process with minimal quasi-stationary distribution $\rho_{\xi_1}$.
We prove that the domain of attraction of $\rho_{\xi_1}$ contains the set of probability measures
\begin{align*}
\cD=\left\{\mu\in\cM_1(\N^*),\, \mu(Q_{\cdot}(\xi_1)=\sum_{i=1}^{\infty} \mu_i Q_i(\xi_1)<+\infty\right\}.
\end{align*}
Once this is proved, the second assertion of Theorem~\ref{thm:minimal-QSD-attraction} follows immediately from Lemma~\ref{lem:useful1}.

\me
Let $\mu$ be a probability measure on $\N^*$ such that $\mu(Q_\cdot(\xi_1))<\infty$.
It is well known (see for instance~\cite[eq. (5.15) and Proposition~5.1]{ColletMartinezsanMartin2013}) that there exists a positive measure $\psi$ whose support is in $[\xi_1,\infty)$ and such that, for any $i,j\geq 1$,
\begin{align}
\label{eq:semi-group_polynomials}
\P_i(X_t=j)=\pi_j\int_{0}^\infty e^{xt} Q_i(x)\,Q_j(x)\,d\psi(x).
\end{align}
But the $\xi_1$-positive recurrence implies that $\psi(\{\xi_1\})>0$ (this two properties are in fact equivalent, see~\cite[Theorem~5.1]{vanDoorn2003}). As a consequence,
\begin{align*}
e^{\xi_1 t}\P_i(X_t=j)\xrightarrow[t\rightarrow\infty]{} Q_i(\xi_1)\psi(\{\xi_1\})\pi_j Q_j(\xi_1)=Q_i(\xi_1)\psi(\{\xi_1\})\frac{d_1\,\rho_{\xi_1}(j)}{\xi_1}.
\end{align*}
Let us define the function $\varphi_1:\N\rightarrow\R_+$ by $\varphi_1(i)=Q_i(\xi_1)$ for all $i\geq 1$ and $\varphi_1(0)=0$. Then, using the equality $\cL \varphi_1=-\xi_1\varphi_1$, we deduce that 
\begin{align*}
e^{\xi_1 t}\P_{\mu}(X_t=j)&\leq e^{\xi_1 t}\P_{\mu}(t< T_0)\\
&\leq e^{\xi_1 t}\E_{\mu}\left(\frac{\varphi_1(X_t) }{\min_{i\geq 1}\varphi_1(i)}\right)\\
&\leq \mu(\varphi_1)<\infty,
\end{align*}
indeed, the minimum of $\varphi_1$ on $\N^*$ is $1$ by Lemma~\ref{lem:usefull2} and $\mu(\varphi_1)=\mu(Q_\cdot(\xi_1))<\infty$ by assumption.

Hence, by dominated convergence theorem, we deduce that
\begin{align*}
 e^{\xi_1 t}\P_{\mu}(X_t=j)\xrightarrow[t\rightarrow\infty]{}  \mu(Q_\cdot(\xi_1))\psi^d(\{\xi_1\})\frac{d_1\,\rho_{\xi_1}(j)}{\xi_1},
\end{align*}
where $\mu(Q_\cdot(\xi_1))=\mu(\phi_1)<\infty$ by assumption.
It is also known that (see~\cite[eq.~(5.21)]{ColletMartinezsanMartin2013})
\begin{align*}
\P_i(t< T_0)=d_1\int_0^\infty \frac{e^{-tx}}{x}Q_i(x)d\psi(x).
\end{align*}
Using the same approach as above, we obtain that
\begin{align*}
e^{\xi_1 t}\P_{\mu}(t<T_0)\xrightarrow[t\rightarrow\infty]{}  \mu(Q_\cdot(\xi_1))\psi^d(\{\xi_1\})\frac{d_1}{\xi_1}.
\end{align*}
We finally deduce the convergence
\begin{align*}
\P_\mu(X_t=j\mid t< T_0)=\frac{e^{\xi_1 t}\P_\mu(X_t=j)}{e^{\xi_1 t}\P_\mu(t< T_0)}\xrightarrow[t\rightarrow\infty]{}
\rho_{\xi_1}(j),
\end{align*}
which means that $\mu$ is in the domain of attraction of the minimal \QSD.
\end{proof}

\section{Approximation of the minimal quasi-stationary distribution}
\label{sec:FV}

This section is devoted to the study of the ergodicity and the convergence of a Fleming-Viot type particle system.

\me Fix $N\geq 2$ and let us describe precisely the dynamics of this system with $N$ particles, which we denote by $(X^1,X^2,\ldots,X^N)$. The process starts at a position $(X^1_0,X^2_0,\ldots,X^N_0)\in\left(\N^*\right)^{N}$ and evolves as follows:
\begin{itemize}
\item[-] the particles $X^i$, $i=1,\ldots,N$, evolve as independent copies of the birth and death process $X$ until one of them hits $0$; this hitting time is denoted by $\tau_1$;
\item[-] then the (unique) particle hitting $0$ at time $\tau_1$ jumps instantaneously on the position of a particle chosen uniformly among the $N-1$ remaining ones; this operation is called a \textit{rebirth};
\item[-]  because of this rebirth, the $N$ particles lie in $\N^*$ at time $\tau_1$; then the $N$ particles evolve as independent copies of $X$ and so on.
\end{itemize}
We denote by $\tau_1<\tau_2<\cdots<\tau_n<\cdots$ the sequence of rebirths times. Since the rate at which rebirths occur is uniformly bounded above by $N\,d_1$,
\begin{align*}
\lim_{n\rightarrow\infty} \tau_n=+\infty\text{ almost surely.}
\end{align*}
As a consequence, the particle system $(X^1_t,X^2_t,\ldots,X^N_t)_{t\geq 0}$ is well defined for any time $t\geq 0$ in an incremental way, rebirth after rebirth (see Figure~\ref{fig:FV2part} for an illustration of this construction with $N=2$ particles).

\me This Fleming-Viot type system has been introduced by Burdzy, Holyst, Ingermann and March
 in \cite{Burdzy1996} and studied in \cite{Burdzy2000}, \cite{Grigorescu2004}, \cite{Villemonais2010}, \cite{Grigorescu2011}
 for multi-dimensional diffusion processes. The study of this system when the underlying Markov process $X$ is a continuous time Markov chain in a countable state space has been initiated in \cite{Ferrari2007} and followed by \cite{Asselah2011}, \cite{Asselah2012}, \cite{Groisman2012}, \cite{AsselahThai2012} and \cite{CloezThai2013}. We also refer the reader to \cite{Groisman2013}, where general considerations on the link between the study of such systems and front propagation problems are considered.
\begin{figure}
\includegraphics[height=7cm]{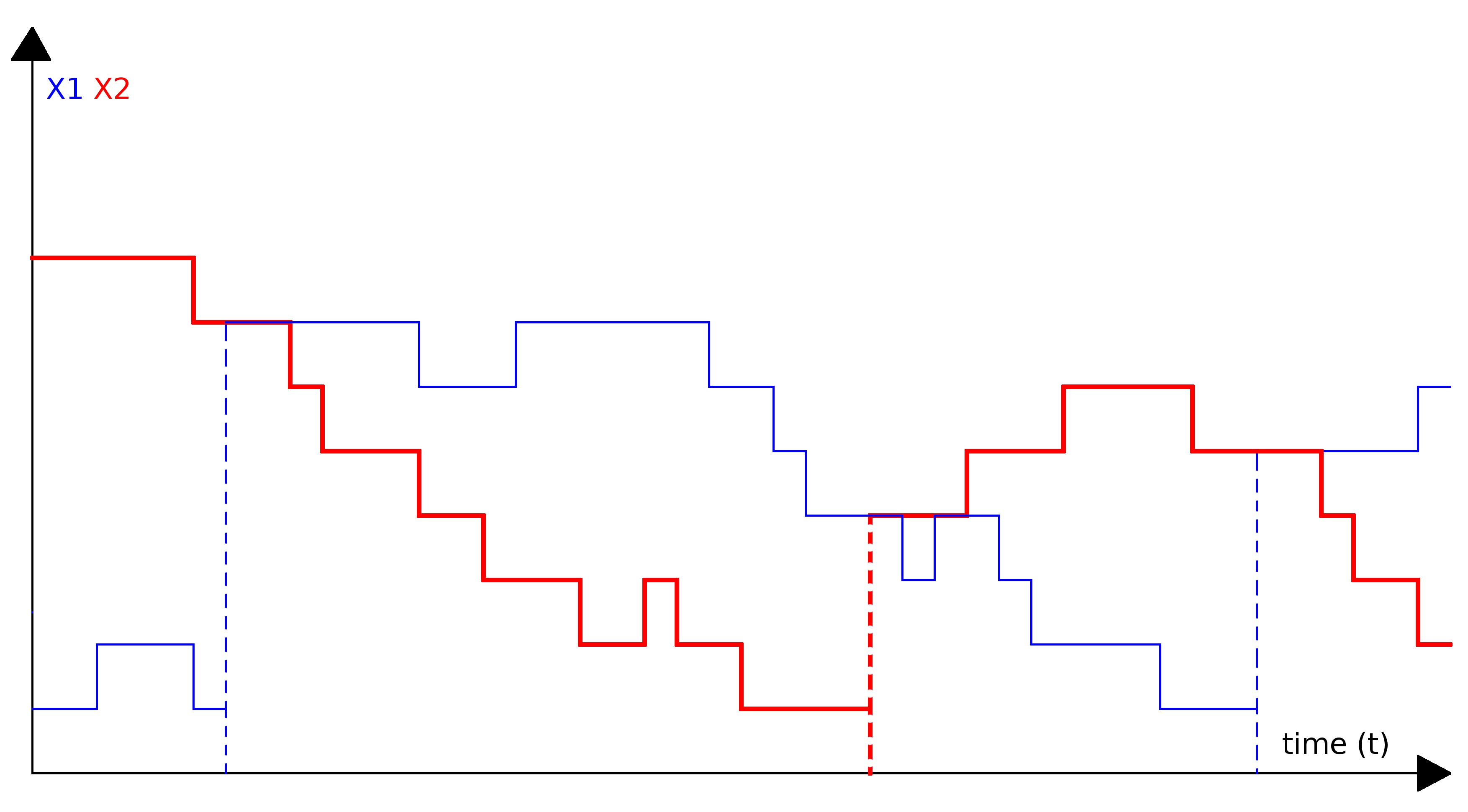}
\caption{One path of a Fleming-Viot system with two particles.}
\label{fig:FV2part}
\end{figure}

\me We emphasize that, because of the rebirth mechanism, the particle system $(X^1,X^2,\ldots,X^N)$ evolves in $\left(\N^*\right)^N$. For any $t\geq 0$, we denote by $\mu^N_t$ the empirical distribution of $(X^1,X^2,\ldots,X^N)$ at time $t$, defined by
\begin{align*}
\mu^N_t=\frac{1}{N}\sum_{i=1}^N\delta_{X^i_t}\in{\cal M}_1(\N^*),
\end{align*}
where ${\cal M}_1(\N^*)$ is the set of probability measures on $\N^*$. A general convergence result obtained in \cite{Villemonais2013} ensures that, if $\mu_0^N\rightarrow\mu_0$, then
\begin{align*}
\mu^N_t\xrightarrow[N\rightarrow\infty]{} \P_{\mu_0}(X_t\in\cdot\mid t<T_0).
\end{align*}

\me The generality of this result does not extend to the long time behaviour of the particle system, which is the subject of the present study. We provide a sufficient criterion ensuring that the process $(\mu^N_t)_{t\geq 0}$ is ergodic. Denoting by $\cX^N$ its empirical stationary distribution (a random measure whose law is the stationary distribution of $\mu^N$), our criterion also implies that 
\begin{align}
\label{eq:intro-main-result}
\cX^N \xrightarrow[N\rightarrow\infty]{Law} \rho,
\end{align}
where $\rho$ is the minimal quasi-stationary distribution of the birth and death process $X$.
Our result applies (1) to birth and death processes with a unique quasi-stationary distribution (such as logistic birth and death processes) and (2) to birth and death processes with a minimal \QSD{} satisfying an explicit Lyapunov condition (fulfilled for instance by linear birth and death processes). These two different conditions are summarized in Assumptions~H1 and~H2 below.

\begin{description}
\item[Assumption H1.] There exist a function $\phi:\N\rightarrow\R_+$ and two constants $\lambda_1>d_1$ and $C\geq 0$ such that $\phi(0)=0$, $\phi(i)>0$ for all $i\geq 1$ and
\begin{align*}
\phi(x)\xrightarrow[x\rightarrow\infty]{} \infty\text{ and }\cL \phi(i)\leq -\lambda_1 \phi(i) + C,\ \forall i\geq 1.
\end{align*}
\item[Assumption H2.] The birth and death process $X$ admits a unique quasi-stationary distribution ($S<+\infty$).
\end{description}

\begin{thm}
\label{thm:main-result}
Assume that Assumption~H1 or Assumption~H2 is satisfied. Then, for any $N> \frac{\lambda_1}{\lambda_1-d_1}$ under H1 and any $N\geq 2$ under H2, the measure process $(\mu^N_t)_{t\geq 0}$ is ergodic, which means that there exists a random measure $\cX^N$ on $\N^*$ such that
\begin{align*}
\mu^N_t\xrightarrow[t\rightarrow\infty]{Law}\cX^N.
\end{align*}
If H1 holds, then $$\E(\phi(\cX^N))\leq C/(\lambda_1-d_1 N/(N-1)).$$
Moreover, if Assumption H1 or H2 is satisfied, then
\begin{align*}
\cX^N\xrightarrow[N\rightarrow\infty]{Law}\rho,
\end{align*}
where $\rho$ is the minimal \QSD{} of $X$.
\end{thm}

\bi

\begin{rem} $\ $
\begin{enumerate}
\item Assumption~H1 is the Lyapunov criterion which is used in Theorem~\ref{thm:positivity_criterion} to ensure $\xi_1$-positivity (and hence the existence of a minimal \QSD). This assumption also implies that the conditions of Theorem~\ref{thm:minimal-QSD-attraction}, where we determine a subset of the domain of attraction of the \QSD, are also satisfied. For instance, the birth and death processes of Examples~1,~2 and~3 in the previous section satisfy Assumption~H1.
\item Assumption~H2 is satisfied for processes that come fast from infinity to compact sets, as the logistic birth and death process (where $b_i=b\,i$ and $d_i=d\,i+c\,i(i-1)$ for all $i\geq 1$ with $b,c,d>0$). Note that, in this particular example, an easy calculation shows that Assumption~H1 is also satisfied with $\phi(i)=2^i$. However, this assumption is useful for any situation where it is easy to check that $S<\infty$, but difficult to find an explicit Lyapunov function satisfying Assumption~H1.
\end{enumerate}
\end{rem}

\begin{rem}
The pure drift birth and death process ($b_i=b$ and $d_i=d$ for all $i\geq 1$, where $b<d$ are two positive constants) does not satisfy Assumption~H1 nor Assumption~H2. Note that this process is the same as in Example~2 but does not satisfy $(\sqrt{d}-\sqrt{b})^2>d_1$. In particular, we cannot apply Theorem~\ref{thm:positivity_criterion} on $\xi_1$-positivity and, in fact, it is known that the pure drift birth and death process is not $\xi_1$-positive recurrent (see~\cite{vanDoorn2003}). As a consequence, the additional difficulty is not a technical one and the following proof cannot work in the pure drift situation. We emphasize that Theorem~\ref{thm:main-result} for pure drift birth and death processes remains an open problem. See for instance~\cite{AsselahThai2012} and the numerical investigation in~\cite{Maric2014} for more details.
\end{rem}

\me
Since the proof of Theorem~\ref{thm:main-result} differs whether one assumes H1 or H2, it is split in two different subsections : in Subsection~\ref{sec:proofH1}, we prove the theorem under Assumption~H1 and, in Subsection~\ref{sec:proofH2}, we prove the result under assumption~H2.

\subsection{Proof under Assumption H1: exponential ergodicity via a Foster--Lyapunov criterion}
\label{sec:proofH1}

\bi \textit{Step 1. Proof of the exponential ergodicity by a Forster--Lyapunov criterion}\\
We define the function 
\begin{align*}
f:\ &\cM_1(\N^*)\rightarrow \R\\
&\mu\mapsto \mu(\phi),
\end{align*}
where $\phi$ is the Lyapunov function of Assumption~H1.
Fix $N\geq 2$ and let us express the infinitesimal generator $\cL^N$ of the empirical process $(\mu^N_t)_{t\geq 0}$ applied to $f$ at a point $\mu\in \cM_1(\N^*)$ given by
\begin{align*}
\mu=\frac{1}{N}\sum_{i=1}^N \delta_{x_i},
\end{align*}
where $(x_1,\ldots,x_N)\in(\N^*)^N$. 
In order to shorten the notations, we introduce, for any $y\in \N^*$, the probability measure
\begin{align*}
\mu^{x_j,y}=\mu+\frac{1}{N}\left(\delta_y-\delta_{x_j}\right).
\end{align*}
We thus have
\begin{align*}
\cL^Nf(\mu)=&\sum_{i=1}^N b_{x_i} \left(f(\mu^{x_i,x_{i+1}})-f(\mu)\right)+\11_{x_i\neq 1} d_{x_i} \left(f(\mu^{x_i,x_{i-1}})-f(\mu)\right)\\
&+\sum_{i=1,x_i=1}^N   d_1\frac{1}{N-1}\sum_{j=1,\,j\neq i}^N \left(f(\mu^{1,x_j})-f(\mu)\right)\\
=&\sum_{i=1}^N b_{x_i} \left(\phi(x_{i+1})-\phi(x_i)\right)/N+\11_{x_i\neq 1}d_{x_i} \left(\phi(x_{i-1})-\phi(x_i)\right)/N\\
&+\sum_{i=1,x_i=1}^N   d_1\frac{1}{N-1}\sum_{j=1,\,j\neq i}^N \left(\phi(x_j)-\phi(1)\right)/N
\end{align*}
Since $\phi(0)=0$, one gets
\begin{align*}
\cL^N f(\mu)=&\sum_{i=1}^N b_{x_i} \left(\phi(x_{i+1})-\phi(x_i)\right)/N+d_{x_i} \left(\phi(x_{i-1})-\phi(x_i)\right)/N\\
&+\sum_{i=1,x_i=1}^N   d_1\frac{1}{N-1}\sum_{j=1,\,j\neq i}^N \phi(x_j)/N\\
=&\frac{1}{N}\sum_{i=1}^N \cL\phi(x_i)+\sum_{i=1,x_i=1}^N   d_1\frac{1}{N-1}\left(\sum_{j=1}^N \phi(x_j)/N-\phi(1)/N\right)\\
\leq &\frac{1}{N}\sum_{i=1}^N \cL\phi(x_i)+\left(\frac{1}{N}\sum_{i=1,x_i=1}^N   d_1\right)\left(\frac{1}{N-1}\sum_{j=1}^N \phi(x_j)\right)\\
\leq &\mu(\cL\phi)+\frac{N}{N-1}d_1\mu(\phi).
\end{align*}
Now, using Assumption~H1, we deduce that
\begin{align*}
\cL^Nf(\mu)&\leq \mu(-\lambda_1\phi+C)+\frac{N}{N-1}d_1\mu(\phi)\\
		   &\leq -\lambda_1 \mu(\phi)+C+\frac{N}{N-1}d_1\mu(\phi)\\
           &\leq -\left(\lambda_1-\frac{N}{N-1}d_1\right)f(\mu)+C,
\end{align*}
where $\lambda_1-\frac{N}{N-1}d_1$ is a positive constant for any fixed $N >\frac{\lambda_1}{\lambda_1-d_1} $.

\noindent For a fixed $N >\frac{\lambda_1}{\lambda_1-d_1}$ and any constant $k>0$, the set of probability measures $\mu=\frac{1}{N}\sum_{i=1}^N\delta_{x_i}$ such that $f(\mu)=\mu(\phi)\leq k$ is finite because $\phi(i)\rightarrow\infty$ when $i\rightarrow\infty$. Moreover the Markov process $(\mu^N_t)$ is irreducible (this is an easy consequence of the irreducibility of the birth and death process $X$). Thus, using the Foster Lyapunov criterion of \cite[Theorem 6.1, p.536]{MeynTweedie1993III} (see also~\cite[Proposition~1.4]{Hairer2010convergence} for a simplified account on the subject), we deduce that the process $\mu^N$ is exponentially ergodic
 and, denoting by $\cX^N$ a random measure distributed following its stationary distribution, we also have
\begin{align}
\label{eq:tightness}
\E(\cX^N(\phi))=\E(f(\cX^N))\leq C/\left(\lambda_1-\frac{N}{N-1}d_1\right).
\end{align}
This concludes the proof of the first part of Theorem~\ref{thm:main-result}.

\bi \textit{Step 2. Convergence to the minimal QSD}\\
Since $\phi(i)$ goes to infinity when $i\rightarrow\infty$, we deduce from~\eqref{eq:tightness} that the family of random measures $(\cX^N)_N$ is tight. 
In particular, the family admits at least one limiting random probability measure $\cX$, which means that $\cX^N$ converges in law to $\cX$, up to a subsequence.

\me Let $\mu^N_t$ be the random position at time $t$ of the particle system with initial (random) distribution $\cX^N$. On the one hand, the 
stationarity of $\cX^N$ implies that $\mu^N_t\sim \cX^N$ for all $t\geq 0$, and thus
\begin{align*}
\mu^N_t\xrightarrow[N\rightarrow\infty]{Law}\cX,\ \forall t\geq 0.
\end{align*}
On the other hand, the general convergence result of~\cite{Villemonais2013} implies that
\begin{align*}
\mu^N_t\xrightarrow[N\rightarrow\infty]{Law} \P_{\cX}(X_t\in\cdot\mid t<\tau_\d).
\end{align*}
As an immediate consequence
\begin{align*}
\P_{\cX}(X_t\in\cdot\mid t<\tau_\d)\stackrel{Law}{=} \cX.
\end{align*}
But~\eqref{eq:tightness} also implies that $\E(f(\cX))<\infty$, so that $\cX(\phi)=f(\cX)<\infty$ almost surely. Using Theorem~\ref{thm:minimal-QSD-attraction}, we deduce that $\cX$ belongs to the domaine of attraction of the minimal QSD $\rho$ almost surely, that is
\begin{align*}
\P_{\cX}(X_t\in\cdot\mid t<\tau_\d)\xrightarrow[t\rightarrow\infty]{almost\, surely} \rho.
\end{align*}
Thus the random measure $\cX$ converges in law to the deterministic measure $\rho$, which implies that
\begin{align*}
\cX=\rho\text{ almost surely}.
\end{align*}
In particular, $\rho$ is the unique limiting probability measure of the family $(\cX^N)_N$, which ends the proof of Theorem~\ref{thm:main-result} under Assumption~H1.

\subsection{Proof under Assumption H2: exponential ergodicity by a Dobrushin coefficient argument}
\label{sec:proofH2}

Fix $N\geq 2$ and let us prove that the process is exponentially ergodic. Under assumption (H2), it is well known (see for instance~\cite{Martinez-Martin-Villemonais2013}) that the process $X$ comes back in finite time from infinity to $1$, which means that
\begin{align*}
\inf_{x\in \N^*} \P_x(X_1=1)>0.
\end{align*}
Since the particles of a Fleming-Viot type system are independent up to the first rebirth time, we deduce that
\begin{align*}
\inf_{(x_1,\ldots,x_N)\in(\N^*)^N} \P((X^1_1,\ldots,X^N_1)=(1,\ldots,1))>0.
\end{align*}
This implies that the FV process is exponentially ergodic.

\me Let us now denotes by $\cX^N$ the empirical stationary distribution of the system $(X^1,\ldots,X^N)$, for each $N\geq 2$. Theorem~\ref{thm:martinez-san-martin-villemonais} implies that there exists $\gamma>0$ such that, for any $t\geq t_{\epsilon}$, any initial distribution $\mu_0$ and any function $f:\N^*\rightarrow\R_+$,
\begin{align*}
\E\left|\rho(f)-\E_{\mu_0}\left(f(X_t)\mid t<\tau_\d\right)\right|\leq 2\gamma^{\lfloor t\rfloor}\|f\|_{\infty}.
\end{align*}
But, for any $t\geq 0$, \cite{Villemonais2013} implies that
\begin{align*}
\E\left|\mu^N_t(f)-\E_{\mu^N_0}\left(f(X_t)\mid t<\tau_\d\right)\right|\leq \frac{2(1+\sqrt{2})e^{d_1 t}\|f\|_{\infty}}{\sqrt{N}}.
\end{align*}
As a consequence,
\begin{align*}
\E\left|\mu^N_t(f)-\rho(f)\right|\leq \frac{2(1+\sqrt{2})e^{d_1 t}\|f\|_{\infty}}{\sqrt{N}}+2\gamma^{\lfloor t\rfloor}\|f\|_{\infty}.
\end{align*}
In particular, for any $\epsilon>0$, there exists $t_\epsilon$ and $N_\epsilon$ such that
\begin{align*}
\E\left|\mu^N_t(f)-\rho(f)\right|\leq \epsilon\|f\|_{\infty},\ \forall N\geq N_\epsilon,\ t\geq t_{\epsilon}.
\end{align*}
But $\mu^N_t$ converges in law to $\cX^N$, so that
\begin{align*}
\E\left|\cX^N(f)-\rho(f)\right|\leq \epsilon\|f\|_{\infty},\ \forall N\geq N_\epsilon.
\end{align*}
This inequality being true for any $\epsilon>0$, this concludes the proof of Theorem~\ref{thm:main-result} under Assumption~(H2).

\section{Numerical simulation of the Fleming-Viot type particle system}
In this section, we present numerical simulations
 of the Fleming-Viot particle system studied in Section~\ref{sec:FV}. Namely, we focus on the distance in total variation norm between the expectation of the empirical stationary distribution (\textit{i.e.} $\E({\cal X}^N)$) and the minimal \QSD{} of the underlying Markov process $X$, when $N$ goes to infinity. This means that we aim at studying the bias of the approximation method.

\me
We start with the linear birth and death process case in Subsection~\ref{sec:LiBDsim}. This is one of the rare situation where explicit computation of the minimal quasi-stationary distribution can be performed (see for instance~\cite{MV12}). In Subsection~\ref{sec:LoBDsim}, we provide the results of numerical simulations in the logistic  birth and death case.

\subsection{The linear birth and death case}
\label{sec:LiBDsim}
We assume in this section that $b_i=i$ and $d_i=2i$ for all $i\geq 0$. This is a sub-case of Example~1 and thus one can apply Theorem~\ref{thm:main-result}: the empirical stationary distribution of the process ${\cal X}^N$ exists and converges in law, when the number $N$ of particles goes to infinity, to the minimal quasi-stationary distribution $\rho$ of the process, which is known to be given by (see~\cite{MV12})
\begin{align*}
\rho(i)=\frac{1}{2^i},\ \forall i\geq 1.
\end{align*}

\me
The results of the numerical estimations of $\|\E({\cal X}^N)-\rho\|_{TV}$ for different values of $N$ (from $2$ to $10^4$) are reproduced on Table~\ref{tab:LiBD}. One interesting point is the confirmation that $\E({\cal X}^N)$ is a biased estimator of $\rho$. A second interesting point is that the bias decreases quickly when $N$ increases. Up to our knowledge, there exists today no theoretical justification of this fact, despite its practical implications. Indeed, one drawback of the speed of the numerical simulation is the interaction between the particles of the Fleming-Viot system: more particles in the system leads to more interaction and thus more communication between processors, which at the end slows down the simulation. A crucial optimisation problem for the approximation method is thus to keep the number of particles as small as possible. In our linear birth and death case, the numerical simulations suggest that the bias decreases as~$O(N^{-1})$.

\renewcommand{\arraystretch}{1.5}
\begin{table}[ht]
\centering
   \begin{tabular}{|c|c|c|}
       \hline
       Nb of particles & $\displaystyle \left\|\E\left({\cal X}^N\right)-\rho\right\|_{TV}$ & Estimated error\\
       \hline
       $N=2$ & $0.190$ & $\pm 10^{-3}$\\
       $N=10$ & $4.5\times 10^{-2}$ & $\pm 10^{-3}$\\
       $N=10^2$ & $5.0\times 10^{-3}$ & $\pm 10^{-4}$\\
       $N=10^3$ & $5.1\times 10^{-4}$ & $\pm 10^{-5}$\\
       $N=10^4$ & $2.3\times 10^{-5}$ & $\pm 10^{-5}$\\
	   \hline
   \end{tabular}
   \caption{\label{tab:LiBD} Estimation of the bias $\|\E({\cal X}^N)-\rho\|_{TV}$ for a linear birth and death process. In the rightmost column, "Estimated error" is the order of magnitude of the error that is made when computing $\|\E({\cal X}^N)-\rho\|_{TV}$.}
\end{table}

\subsection{The logistic birth and death case}
\label{sec:LoBDsim}
We consider now the case where $b_i=2i$ and $d_i=i+i(i-1)$, for all $i\geq 1$. The existence and uniqueness of a \QSD{} $\rho$ is well known for this process, but no explicit formula for the probability measure $\rho$ exists. Thus, in order to compute numerically the total variation distance $\|\E({\cal X}^N)-\rho\|_{TV}$ for different values of $N$, we use the approximation $$\rho\simeq \E({\cal X}^{N_0})\text{, where }N_0=10^4.$$ The histogram of the estimated quasi-stationary distribution is represented on Figure~\ref{fig:estimated-rho}.
\begin{figure}[ht]
\begin{center}
\includegraphics[width=10cm]{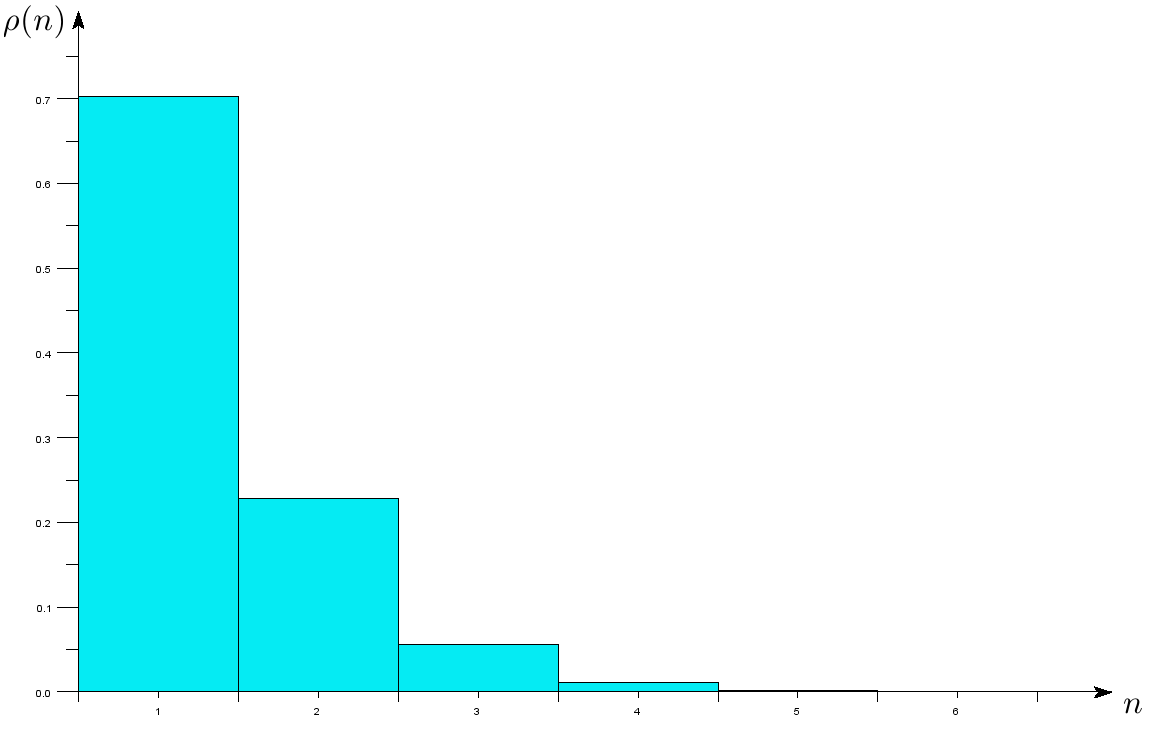}
\end{center}
\caption{\label{fig:estimated-rho} Estimated value of the minimal \QSD{} $\rho(n)$ for a logistic birth and death process.}
\end{figure}

\noindent The results of the numerical estimations of the bias $\|\E({\cal X}^N)-\rho\|_{TV}$ are reproduced on Table~\ref{tab:LoBD}. The conclusion is the same as in the linear birth and death case : $\|\E({\cal X}^N)-\rho\|_{TV}$ declines very sharply as a function of $N$. In fact, the phenomenon is even more spectacular, since the estimated value of $\left\|\E\left({\cal X}^N\right)-\rho\right\|_{TV}$ is $2.0\times 10^{-2}$, even for $N=2$.

\renewcommand{\arraystretch}{1.5}
\begin{table}[ht]
\centering
   \begin{tabular}{|c|c|c|}
       \hline
       Nb of particles & $\displaystyle \left\|\E\left({\cal X}^N\right)-\rho\right\|_{TV}$ & Estimated error\\
       \hline
       $N=2$ & $2.0\times 10^{-2}$ & $\pm 10^{-3}$\\
       $N=10$ & $3.0\times 10^{-3}$ & $\pm 10^{-4}$\\
       $N=10^2$ & $3.6\times 10^{-4}$ & $\pm 10^{-5}$\\
       $N=10^3$ & $2\times 10^{-5}$ & $\pm 10^{-5}$\\
       $N=10^4$ & $***$ & $\pm 10^{-5}$\\
	   \hline
   \end{tabular}
   \caption{\label{tab:LoBD} Estimation of the bias $\|\E({\cal X}^N)-\rho\|_{TV}$ for a logistic birth and death process. In the rightmost column, "Estimated error" is the order of magnitude of the error that is made when computing $\|\E({\cal X}^N)-\rho\|_{TV}$. On the last line, there is no input for $N=N_0=10^4$ because $\rho\simeq \E({\cal X}^{N_0})$ is the probability measure used to compute $\|\E({\cal X}^N)-\rho\|_{TV}$.}
\end{table}

\bibliographystyle{plain}
\bibliography{biblio-bio,biblio-math,biblio-denis}

\end{document}